\documentclass[a4paper,11pt, leqno]{article}
\usepackage[T1]{fontenc}
\usepackage[latin9]{inputenc}
\usepackage{verbatim}
\usepackage{mathrsfs}
\usepackage{amsthm}
\usepackage{amsmath}
\usepackage{amssymb}
\usepackage[all]{xy}
\usepackage{ifthen}
\usepackage{mathdots}
\usepackage{enumerate}
\usepackage{blkarray}
\usepackage{multirow}

\theoremstyle{plain}

\newtheorem*{thm1}{Theorem 1}

\newtheorem*{corollary1}{Corollary 1}
\newtheorem*{corollary2}{Corollary 2}
\newtheorem*{corollary3}{Corollary 3}
\newtheorem{theorem}{Theorem}[section]
\newtheorem{lemma}[theorem]{Lemma}

\newtheorem{proposition}[theorem]{Proposition}

\newtheorem{prop}[theorem]{Proposition}

\newtheorem{intro-theorem}{Theorem}
\newtheorem{intro-corollary}[intro-theorem]{Corollary}
\newtheorem{intro-proposition}[intro-theorem]{Proposition}

\theoremstyle{definition}

\newtheorem{remark}[theorem]{Remark}

\numberwithin{equation}{section}
\numberwithin{figure}{section}

\newskip\procskipamount
\newskip\interskipamount
\newskip\refskipamount
\newcommand{\procskip}{\vskip\procskipamount}
\newcommand{\interskip}{\vskip\interskipamount}
\newcommand{\refskip}{\vskip\refskipamount}

\newcommand{\procbreak}{\par
   \ifdim\lastskip<\procskipamount\removelastskip
   \penalty-100
   \procskip\fi
   \noindent\ignorespaces}

\newcommand{\titlebreak}{\par%
\ifdim\lastskip<\interskipamount\removelastskip%
\penalty10000%
\interskip\fi%
\noindent}%

\newcommand{\interbreak}{\par%
\ifdim\lastskip<\interskipamount\removelastskip%
\penalty-100%
\interskip\fi%
\noindent\ignorespaces}%

\newcommand{\refbreak}{\par%
\ifdim\lastskip<\refskipamount\removelastskip%
\penalty-100%
\refskip\fi%
\noindent\ignorespaces}%

\newcounter{refcounter}

{\par}%

\newcounter{listcounter}
\newcounter{deflistcounter}
\newcounter{equivcounter}

\newskip{\itemsepamount}
\newskip{\topsepamount}
\newenvironment{assertionlist}{%
  \begin{list}
    {\upshape (\arabic{listcounter})}
    {\setlength{\leftmargin}{18pt}
     \setlength{\rightmargin}{0pt}
     \setlength{\itemindent}{0pt}
     \setlength{\labelsep}{5pt}
     \setlength{\labelwidth}{13pt}
     \setlength{\listparindent}{\parindent}
     \setlength{\parsep}{0pt}
     \setlength{\itemsep}{\itemsepamount}
     \setlength{\topsep}{\topsepamount}
     \usecounter{listcounter}}}
  {\end{list}}

\newenvironment{definitionlist}{%
  \begin{list}
    {\upshape (\alph{deflistcounter})}
    {\setlength{\leftmargin}{18pt}
     \setlength{\rightmargin}{0pt}
     \setlength{\itemindent}{0pt}
     \setlength{\labelsep}{5pt}
     \setlength{\labelwidth}{13pt}
     \setlength{\listparindent}{\parindent}
     \setlength{\parsep}{0pt}
     \setlength{\itemsep}{\itemsepamount}
     \setlength{\topsep}{\topsepamount}
     \usecounter{deflistcounter}}}
  {\end{list}}

\newenvironment{equivlist}{%
  \begin{list}
    {\upshape (\roman{equivcounter})}
    {\setlength{\leftmargin}{18pt}
     \setlength{\rightmargin}{0pt}
     \setlength{\itemindent}{0pt}
     \setlength{\labelsep}{5pt}
     \setlength{\labelwidth}{13pt}
     \setlength{\listparindent}{\parindent}
     \setlength{\parsep}{0pt}
     \setlength{\itemsep}{\itemsepamount}
     \setlength{\topsep}{\topsepamount}
     \usecounter{equivcounter}}}
  {\end{list}}

% ...

\renewcommand{\AA}{\mathbb{A}}

\newcommand{\CC}{\mathbb{C}}

\newcommand{\FF}{\mathbb{F}}
\newcommand{\GG}{\mathbb{G}}

\newcommand{\QQ}{\mathbb{Q}}
\newcommand{\RR}{\mathbb{R}}

\newcommand{\ZZ}{\mathbb{Z}}

\renewcommand{\hbar}{\bar{h}}

\newcommand{\Gbf}{{\bf G}}

\newcommand{\Xbf}{{\bf X}}

\newcommand{\Ecal}{{\mathcal E}}

\newcommand{\Gcal}{{\mathcal G}}

\newcommand{\Ocal}{{\mathcal O}}

\newcommand{\Scal}{{\mathcal S}}

\newcommand{\Ascr}{{\mathscr A}}

\newcommand{\Lscr}{{\mathscr L}}

\newcommand{\Ktilde}{\tilde{K}}

\DeclareMathOperator{\Gal}{Gal}
\DeclareMathOperator{\GL}{GL}

\DeclareMathOperator{\GSp}{GSp}

\DeclareMathOperator{\id}{id}

\DeclareMathOperator{\Norm}{Norm}

\newcommand{\One}{\hbox{\rm1\kern-2.3ptl}}

\DeclareMathOperator{\Pic}{Pic}

\newcommand\addots{\mathinner{\mkern1mu\raise0pt\vbox{\kern7pt\hbox{.}}\mkern2mu\raise3pt\hbox{.}\mkern2mu\raise6pt\hbox{.}\mkern1mu}}
\newcommand{\bs}{\backslash}

\newcommand{\leftexp}[2]{{\vphantom{#2}}^{#1}{#2}}

\newcommand{\set}[2]{\{\,#1\ ;\  #2\,\}}

\newcommand{\vdual}{{}^{\vee}}

\newcommand\lto{\longrightarrow}

\newcommand\varto[1]{\mathrel{\hbox to #1pt{\rightarrowfill}}}

\newcommand{\mono}{\hookrightarrow}

\newcommand{\sends}{\mapsto}

\newcommand{\liso}{\overset{\sim}{\lto}}

\newcommand{\GZip}{\mathop{\text{$G$-{\tt Zip}}}\nolimits}

\begin{document}

\title{Sections of the Hodge bundle over Ekedahl-Oort strata of Shimura varieties of Hodge type}

\author{Jean-Stefan Koskivirta}

\maketitle

\abstract{We construct canonical non-vanishing global sections of powers of the Hodge bundle on each Ekedahl-Oort stratum of a Hodge type Shimura variety. In particular we recover the quasi-affineness of the Ekedahl-Oort strata. In the projective case, this gives a very short proof of non-emptiness of Ekedahl-Oort strata. It follows that the Newton strata are also nonempty, by a result of S.Nie. We also deduce the fact that the $\mu$-ordinary locus is determined by the Ekedahl-Oort strata of its image by any embedding.}

\section*{Introduction}

The Siegel modular variety $\mathcal{A}_{g,N}$ arises as a moduli space of $g$-dimensional principally polarized abelian varieties with level-$N$ structure. More generally, Shimura varieties of PEL-type classify abelian varieties endowed with a polarization, an action of a semisimple algebra, and a level structure. These varieties satisfy nice properties due to there nature as moduli spaces. For example, the special fiber of a PEL-Shimura variety at a place of good reduction carries different stratifications (Ekedahl-Oort, Newton, $p$-rank, etc.).
\medskip 

Shimura varieties which can be embedded into a Siegel modular variety are called of Hodge type. In general, they do not have an interpretation as a moduli space. In particular, all PEL-Shimura varieties are of Hodge type. For Hodge type varieties, one can still define the Ekedahl-Oort stratification using the stack of $G$-zips, introduced by Wedhorn, Moonen, Pink, Ziegler in \cite{MoWd}, \cite{PWZ1} and \cite{PWZ2}. More precisely, let $(\Gbf,\mu)$ be a Hodge type Shimura datum and let $S_K$ denote the special fiber of the associated Shimura variety at a place of good reduction. Zhang has constructed in \cite{Zhang_EO} a $G$-zip over $S_K$ and he has shown that the corresponding morphism
$$S_K \to \GZip^{\mu}$$
is smooth. By definition, the Ekedahl-Oort strata of $S_{K}$ are the fibers of $\zeta$. In particular, Zhang proves that the $\mu$-ordinary Ekedahl-Oort stratum is open and  dense. It is also possible to define group theoretically the Newton stratification on $S_K$. Wortmann has shown in \cite{Wort_MuOrd} that the open Newton stratum coincides with the $\mu$-ordinary locus.
\medskip 

Fix an embedding $\iota:S_K\to \mathcal{A}_{g,N}$ of a Hodge type Shimura variety $S_K$ into a Siegel modular variety. Denote by $\Ascr \to S_K$ the pull-back of the universal abelian scheme on $\mathcal{A}_{g,N}$ via $\iota$. Let $e:S_K\to \Ascr$ denote the identity section of $\Ascr$. Then the Hodge bundle is by definition
$$\omega_S:=\det(e^*\Omega_{\Ascr/S_K}).$$
This line bundle is ample on $S_K$. In an upcoming paper by the author and T.Wedhorn, it is proved that this line bundle admits a canonical global section, a generalized Hasse invariant, which vanishes exactly outside the $\mu$-ordinary locus (\cite{KoWd} Theorem 4.12). In this paper we construct sections of $\omega_S$ on each Ekedahl-Oort stratum. Here are our main results:

\begin{thm1} \label{thm1}
Let $S^w$ be a nonempty Ekedahl-Oort stratum. There exists $N_w \geq 1$ such that for all $d\geq 1$, there exists a (canonical) non-vanishing section in the space
$$H^{0}(S^w,\omega_S^{\otimes N_w d}).$$
\end{thm1}

This section is canonical, in the sense that it is a pullback of a non-vanishing global section on a certain substack of the stack of $G$-zips. Therefore it only depends on the choice of the embedding into a Siegel Shimura variety. This non-vanishing section induces an isomorphism $\Ocal_{S^w}\to \omega_S^{\otimes N_w d}$. In other words, Theorem 1 says that the line bundle $\omega_S$ is torsion on $S^w$. This implies that $\Ocal_{S^w}$ is ample on $S^w$, so we recover the following result:

\begin{corollary1}
The Ekedahl-Oort strata are quasi-affine.
\end{corollary1}

As pointed out by Ulrich Görtz and Chia-Fu Yu, this result can also be deduced immediately from the quasi-affineness of the Ekedahl-Oort strata of the Siegel Shimura variety: Each stratum $S^w$ is locally closed in the preimage of the corresponding Siegel Ekedahl-Oort stratum, and is therefore quasi-affine. Let us also mention that by a result of Wedhorn and Yaroslav, the inclusion $S^w\to S$ is an affine morphism (\cite{WdYa} Theorem 4.1). From the quasi-affiness of Ekedahl-Oort strata, we deduce the following:

\begin{corollary2}\label{nonemp}
Let $S_K$ be a Hodge type Shimura variety. Assume that $S_K$ is projective. Then all Ekedahl-Oort and Newton strata are nonempty.
\end{corollary2}

For PEL-Shimura varieties, all Ekedahl-Oort strata are known to be nonempty, by a result of Viehmann and Wedhorn in \cite{ViWd}. For more general Hodge type Shimura varieties, nonemptiness is expected to hold, even though no proof has been given so far.

\medskip

Here is an idea of the proof of Corollary 2. The ampleness of the Hasse bundle and the existence of the Hasse invariant imply that an Ekedahl-Oort of positive dimension cannot be projective. Using an inductive argument, we deduce that the superspecial locus (the Ekedahl-Oort stratum of dimension zero) is nonempty, and then the result is a consequence of the flatness of the map $\zeta$. Theorem 1 has also the following consequence: 

\begin{corollary3} \label{mu}
Let $S^\mu$ be the $\mu$-ordinary Ekedahl-Oort stratum, and let $S_0^\mu$ be the Ekedahl-Oort stratum of the Siegel modular variety $\mathcal{A}_{g,N}$ containing the image of $S^\mu$ by the embedding $\iota:S_K\to \mathcal{A}_{g,N}$. Then one has the equality:
$$S^\mu=\iota^{-1}(S_0^\mu).$$
\end{corollary3}

For example, assume $S$ is a PEL Shimura variety parametrizing tuples $(A,\lambda,\iota,\bar{\nu})$ where $A$ is an abelian variety, $\lambda$ a polarization, $\iota$ an action of a $\ZZ_{(p)}$-algebra, and $\bar{\nu}$ a level structure. Then the $\mu$-ordinary locus of $S$ is entirely determined by the isomorphism class of the $p$-torsion $A[p]$, forgetting the rest of the structure.

We now give an overview of the paper. In the first section we recall the parametrization of Ekedahl-Oort strata using the stack of $G$-zips and the map $\zeta$. Then in the second part we state some general facts about the Picard group of a quotient stack and the space of global sections of a line bundle, which we apply to the stack of $G$-zips. In the third part we construct Hasse invariants on each Ekedahl-Oort stratum. Finally, we prove the corollaries of Theorem 1 in the last subsection.

\subsection*{Acknowledgements}
I would like to thank Torsten Wedhorn for his very useful comments and fruitful conversations. I am also grateful to Ulrich Görtz for some remarks on the paper.

\section{Parametrization of Ekedahl-Oort strata}

\subsection{Shimura varieties of Hodge-type}

We follow the general setup of \cite{KoWd} (4.1). Let $(\Gbf,\Xbf)$ be a Shimura datum of Hodge type as in \cite{De_VarSh}. We denote by $[\mu]$ the $\Gbf(\CC)$-conjugacy class of the component of $h_{\CC}\colon \prod_{\Gal(\CC/\RR)}\GG_{m,\CC} \to \Gbf_{\CC}$ corresponding to $\id \in \Gal(\CC/\RR)$. Let $E$ be the reflex field, i.e the field of definition of $[\mu]$.

We fix an embedding $\iota:(\Gbf,\Xbf)\to (\GSp(V),S^{\pm})$ of Shimura datum, where $V = (V,\psi)$ is a symplectic space over $\QQ$ and $S^{\pm}$ the double Siegel half space. Let $p$ be a prime of good reduction $\Gcal$ a reductive $\ZZ_{(p)}$-model of $\Gbf$, and $K_p := \Gcal(\ZZ_{p})$ the hyperspecial subgroup of $\Gbf(\QQ_p)$. We denote by $G$ the special fiber of $\Gcal$.

Choose a place $v$ of the reflex field $E$ of $(\Gbf,\Xbf)$ over $p$. Let $K = K_pK^p \subset \Gbf(\AA_f)$ be a compact open subgroup. We assume $K^p$ sufficiently small, so that an integral canonical model $\Scal_K(\Gbf,\Xbf)$ over $O_{E,v}$ exists for the Shimura variety attached to $(\Gbf,\Xbf)$ (see \cite{Ki_Integral} and \cite{Vasiu_Integral}). We denote by $S := S_K(\Gbf,\Xbf)$ the special fiber. It is a smooth quasi-projective scheme over the residue field $\kappa := \kappa(v)$. Let $k$ denote the algebraic closure of $\kappa$.

The embedding $\iota$ admits a model over $\ZZ_{(p)}$. More precisely, there is a $\ZZ_{(p)}$-lattice $\Lambda$ of $V$ such that $\iota$ is induced by an embedding $\Gcal \to \GL(\Lambda)$ (\cite{Ki_Integral}~Lemma~(2.3.1)). We may assume that $\psi$ induces a perfect $\ZZ_{(p)}$-pairing on $\Lambda$ (\cite{Ki_Points}~(1.3.3)). We obtain an embedding
\begin{equation}\label{EqEmbGroup}
\iota\colon \Gcal \mono \GSp(\Lambda)
\end{equation}
over $\ZZ_{(p)}$ such that $\Gcal$ is the scheme theoretic stabilizer of a finite set $s$ of tensors in $\Lambda^{\otimes}$. We then can extend the embedding of Shimura varieties in characteristic zero to an embedding
\begin{equation}\label{emb}
\Scal_K(\Gbf,\Xbf)\mono \Scal_{\Ktilde}(\GSp(\Lambda),S^{\pm}) \otimes_{\ZZ_{(p)}} O_{E,v}
\end{equation}
where $\Ktilde := \Ktilde_p\Ktilde^p$, $\Ktilde_p := \GSp(\Lambda)(\ZZ_p)$, and $\Ktilde^p \subset \GSp(\AA^p_f)$ is a certain open compact subgroup. Let $\tilde\Ascr \to \Scal_{\Ktilde}(\GSp(\Lambda),S^{\pm})$ be the universal abelian scheme and let $\Ascr$ be its pullback to $\Scal := \Scal_K(\Gbf,\Xbf)$ via the embedding \eqref{emb}. We define:
\[
\omega := \det(e^*\Omega^1_{\Ascr/\Scal})
\]
where $e$ is the identity section of $\Ascr$, and we call it the \emph{Hodge line bundle on $\Scal$}. The line bundle $\omega$ is ample (\cite{KoWd} Proposition 4.1). We denote by $\omega_S$ the pullback of $\omega$ to the special fiber.

\subsection{The stack of $G$-zips}

The conjugacy class $[\mu^{-1}]$ extends to a conjugacy class over $O_{E_v}$. As $\Gcal$ is quasi-split, there exists a representative defined over $O_{E_v}$. We denote by $\chi$ the reduction over $\kappa$ of this representative. Let $P_{\pm} = P_{\pm}(\chi)$ be the pair of opposite parabolic subgroups of $G_{\kappa}$ attached to the cocharacter $\chi$, with common Levi subgroup $L$ (the centralizer of $\chi$). Then $(G,P_+,\sigma(P_-),\varphi)$ is an algebraic zip datum in the sense of \cite{PWZ1}~10.1, where $\sigma(-)$ denotes the pullback under absolute Frobenius $\sigma\colon x \sends x^p$ and where $\varphi\colon L \to \sigma(L)$ is the relative Frobenius. We set $P:=P_+$, $Q:=\sigma(P_-)$ and $M:=\sigma(L)$, so that $M$ is the Levi subgroup of $Q$ containing $T$.

We may assume, possibly after conjugating $\chi$ over $\kappa$, that there is a Borel pair $(T,B)$ defined over $\FF_p$ such that $B_-\subset P$, and therefore $B\subset Q$ (see \cite{KoWd} Lemma 4.2).
Let $(X,\Phi,X\vdual,\Phi\vdual,\Delta)$ be the based root datum of $(G,B,T)$. Denote by $W = W(G,T) := N_{G}(T)/T$ the Weyl group and by $I \subset W$ the set of simple reflections defined by $B$. The subsets of $I$ correspond bijectively to the parabolic subgroups containing $B$, which are called \emph{standard}. For $J \subset I$, denote by $Q_{J}$ the corresponding standard parabolic and by $M_{J}$ the unique Levi subgroup of $Q_{J}$ containing $T$. We have an inclusion $W_J := W(M_{J},T)\hookrightarrow W(G,T)$
such that $J = W_J \cap I$. Every parabolic subgroup $P'$ of $G$ is conjugate to a unique standard parabolic subgroup $Q_J$ and $J \subset I$ is called the \emph{type of $P'$}.

For $x\in P$, we denote by $\overline{x}$ the image of $x$ in $P/R_{u}(P) = L$, and similarly for the image of $y \in Q$ in $Q/R_u(Q) = M$. The associated \emph{zip group} is defined by
\[
E := \set{(x,y)\in P\times Q}{\varphi(\overline{x})=\overline{y}} 
\]
and $E$ acts on $G$ by $(x,y)\cdot g:= xgy^{-1}$. Note that $\dim(E) = \dim(G)$. By \cite{PWZ1}~Proposition~7.3, there are finitely many $E$-orbits in $G$, which are parametrized as follows. Let $J \subseteq I$ and $K \subseteq I$ be the type of $P$ and $Q$, respectively. For every $w \in W$ we choose a representative $\dot w \in \Norm_G(T)$ such that $(w_1w_2)^{\cdot} = \dot w_1\dot w_2$ whenever $\ell(w_1w_2) = \ell(w_1) + \ell(w_2)$. Let $w_{0,J} \in W_J$ and $w_0 \in W$ be the longest elements and set $g_0 := (w_{0} w_{0,J})^{\cdot}$. By \cite{PWZ1}~Theorem~5.12 and Theorem~7.5 we obtain a bijection
\begin{equation}\label{EqEOrbits}
\leftexp{J}{W} \liso \{\text{$E$-orbits on $G$}\}, \qquad w \sends O^w := E\cdot (g_0\dot{w})
\end{equation}
such that $\dim O^w = \ell(w) + \dim(P)$.

\subsection{Ekedahl-Oort strata}
The algebraic quotient stack over $\kappa$
\begin{equation}\label{EqDefGZipStack}
\GZip^{\chi} := [E \bs G_{\kappa}]
\end{equation}
is called the \emph{stack of $G$-zips}. The underlying topological space of $\GZip^{\chi}$ is homeomorphic to $\leftexp{J}{W}$ endowed with the order topology with respect to a certain partial order $\preceq$; see \cite{PWZ1}~Definition~6.1 for the precise definition.

Zhang has constructed in \cite{Zhang_EO} a $G$-zip of type $\chi$ over $S_K := S_K(\Gbf,\Xbf)$ and he has shown in loc.~cit.\ that the corresponding morphism $S_K \to \GZip^{\chi}$ is smooth. In this paper, we prefer to use the construction given by Wortmann in \cite{Wort_MuOrd}~\S5 and we obtain again a smooth morphism
\begin{equation}\label{EqDefineZeta}
\zeta := \zeta_{G}\colon S_K \lto \GZip^{\chi}.
\end{equation}
The Ekedahl-Oort strata of $S_{K}$ are defined as the fibers of $\zeta$. For $w\in \leftexp{J}{W}$, we denote by $S^w:=\zeta^{-1}(w)$ the corresponding stratum endowed with the reduced scheme structure as a locally closed subset of $S_K$. Then $S^w$ is smooth by \cite{WdYa} and if nonempty, has dimension $\ell(w)$. In the case of PEL Shimura varieties, every Ekedahl-Oort stratum is nonempty (\cite{ViWd}~Theorem~10.1). The map (\ref{EqDefineZeta}) restricts to a smooth map of stacks:
\begin{equation}\label{Zetaw}
\zeta \colon S^w \lto [E \bs O^w].
\end{equation}

\section{Equivariant Picard group}
\subsection{$G$-linearizations}
In this section we consider an arbitrary smooth algebraic group over $k$ acting on a $k$-variety $X$. If $\pi:L\to X$ is a line bundle, a $G$-linearization of $L$ is a map
$$G\times L\to L$$
defining an action of $G$ on $L$, satisfying the conditions:
\begin{enumerate}[(i)]
\item The map $\pi$ is $G$-equivariant.
\item The action of $G$ on $L$ is linear on the fibers.
\end{enumerate}

We denote by $\Pic^{G}(X)$ the group of isomorphism classes of $G$-linearized line bundles on $X$. The image of the natural map $\Pic^{G}(X) \to \Pic(X)$ is the subgroup of $G$-linearizable line bundles, and is denoted by $\Pic_{G}(X)$. The group $\Pic^{G}(X)$ can be identified with the Picard group of the quotient stack $\left[G \bs X\right]$.

We define $\Ecal(X):=\frac{\GG_m(X)}{k^\times}$. If $X$ is an irreducible variety over $k$, it is a free abelian group of finite type.

\begin{prop}\label{exseqgen}
Let $G$ be a smooth algebraic group, and $X$ an irreducible $G$-variety. Then
there is an exact sequence:
\[
1 \to k^{\times} \to \GG_m(X)^G \to \Ecal(X) \to X^{*}(G) \to \Pic^{G}(X) \to \Pic(X)
\]
\end{prop}

\begin{proof}
See \cite{KKV}, Proposition 2.3 and Lemma~2.2. The assumption that $k$ is of characterstic $0$ is not needed in the proof.
\end{proof}

The map $X^*(G)\to \Pic^{G}(X)$, $\lambda \mapsto \Lscr(\lambda)$ is defined as follows. A character $\lambda \in X^{*}(G)$ induces a $G$-linearization of the trivial line bundle $\AA^1_k \times X$ on $X$ given by $(g,x,s)\mapsto(g\cdot x,\lambda(g)s)$ for all $g\in G$, $x\in X$, $s\in \AA^1_k$.

\begin{proposition} \label{charpic}
Let $H \subset G$ be algebraic groups. Then there is a natural isomorphism:
\[
\Pic^G(G/H)\simeq X^*(H).
\]
\end{proposition}

\begin{proof}
One has $\Pic^G(G/H)\simeq \Pic([G\backslash (G/H)])\simeq \Pic([1/H])\simeq \Pic^H(1)\simeq X^*(H)$.
\end{proof}

\subsection{The space of global sections}

\begin{prop}
Let $G$ be an algebraic group and let $X$ be an irreducible $G$-variety containing an open $G$-orbit $U$. Denote by $\pi:X\to\left[X/G\right]$
the projection. Let $\mathscr{L}$ be a line bundle on the stack $\left[X/G\right]$ and write $L=\pi^{*}\mathscr{L}$. Then:
\end{prop}

\begin{enumerate}
\item[\textit{(i)}] \textit{$H^{0}\left(\left[X/G\right],\mathscr{L}\right)$ identifies with $H^{0}\left(X,L\right)^{G}$. In particular, for $\lambda \in X^*(G)$ one has:}
$$H^{0}\left(\left[X/G\right],\Lscr(\lambda)\right)=\{f:X\to k, f(g\cdot x)=\lambda(g)f(x), \ \forall g\in G,x\in X \}.$$
\item[\textit{(ii)}] \textit{The $k-$vector space $H^{0}\left(\left[X/G\right],\mathscr{L}\right)$
has dimension less than 1.}
\item[\textit{(iii)}] \textit{If $H^{0}\left(\left[X/G\right],\mathscr{L}\right)\neq0$ then $\mathscr{L}$
restricts to the trivial line bundle on $\left[U/G\right]$.}
\item[\textit{(iv)}] \textit{If $\mathscr{L}$ is trivial, $H^{0}\left(\left[X/G\right],\mathscr{L}\right)=k$.}
\end{enumerate}

\begin{proof}
See \cite{KoWd} Proposition 1.18.
\end{proof}

\section{Hasse invariants on Ekedahl-Oort strata}

\subsection{Construction}

In this section we construct a canonical non-vanishing section of the Hodge bundle on each Ekedahl-Oort stratum of $S_K$. We do not know if this section extends to the closure, and if it does, what its vanishing locus would be. Since we expect this to be true, we call abusively this section a Hasse invariant of the stratum. In the particular case of the $\mu$-ordinary stratum, it was proved in \cite{KoWd} Theorem 4.12, that this canonical section does extand to the whole Shimura variety (and even to its minimal compactification), and that the non-vanishing locus is exactly the $\mu$-ordinary stratum.

Let $G$ be a reductive group over $\FF_p$, and $\mu:\GG_{m,k}\to G_k$ a minuscule cocharacter. Denote by $(G,P,Q,\varphi)$ the associated zip datum, and by $E$ the attached zip group.

\begin{theorem}\label{H0}
For all $E$-orbit $C\subset G$, the Picard group $\Pic^E(C)$ is finite. Denote by $N_C$ its exponent. Then for all $d\geq 1$ and all $\lambda \in X^*(E)$, the space of global sections
$$H^{0}\left(\left[E \backslash C\right],\Lscr(\lambda)^{\otimes N_C d}\right)$$
is one-dimensional.
\end{theorem}

\begin{proof}
We apply Proposition \ref{exseqgen} to the $E$-variety $C$. Clearly $\GG_m(C)^E=k^\times$. Hence we get an exact sequence:
\begin{equation} \label{exseqproof}
1 \to \Ecal(C) \to X^{*}(E) \to \Pic^E(C) \to \Pic(C)
\end{equation}
Let $x$ be an arbitrary element of $C$. The map $E\to C$, $e\mapsto e\cdot x$ identifies $C$ with the quotient $E/A_x$, where $A_x$ is the scheme-theoretic stabilizer of $x$ in $E$. We have an exact sequence
$$1\to A_{x,{\rm red}}\to A_x \to A_x /A_{x,{\rm red}} \to 1$$
where $A_x /A_{x,{\rm red}}$ is a finite group scheme. Hence we have an exact sequence
$$0\to X^*(A_x /A_{x,{\rm red}}) \to X^*(A_x)\to X^*(A_{x,{\rm red}}).$$
By \cite{PWZ1} Theorem 8.1, the group $A_{x,{\rm red}}$ has the form $U_x \rtimes H_x$ where $U_x$ is unipotent and $H_x$ finite. We deduce that $X^*(A_{x,{\rm red}})$ is finite, and hence so is $X^*(A_x)$. It follows from Proposition \ref{charpic} that $\Pic^E(C)$ is finite. Let $N_C$ be its exponent.

For all $d\geq 1$, the character $N_C d \lambda$ maps to zero in $\Pic^E(C)$. Therefore there exists a function $f\in \Ecal(C)$ mapping to $N_C d \lambda$. By definition, this is a non-vanishing function on $C$ satisfying the relation $f(e\cdot x)=\lambda(e)^{N_C d}f(x), \ \forall e\in E,x\in C$, so it is a nonzero global section of $\Lscr(\lambda)^{\otimes N_C d}$. This concludes the proof.
\end{proof}

\begin{remark}\label{rmkN}
For a fixed character $\lambda \in X^*(E)$, let $N_C(\lambda)$ be the order of $\Lscr(\lambda)$ in $\Pic^E(C)$. The set of integers $r$ such that $H^{0}\left(\left[E \backslash C\right],\Lscr(\lambda)^{\otimes r}\right)\neq 0$ is the subgroup of $\ZZ$ generated by $N_C(\lambda)$.
\end{remark}

The first projection $E\to P$ induces an isomorphism $X^*(E)=X^*(P)$. A character $\lambda\in X^*(E)=X^*(P)$ is said to be \emph{ample} if the associated line bundle on $G/P$ is ample, see Definition 3.2 in \cite{KoWd}. This defines a cone in $X^*(E)$. The following result is a reformulation of Theorem 3.8 in loc. cit.

\begin{theorem}\label{mainthm}
Let $U\subset G$ denote the open $E$-orbit of $G$. Let $\lambda \in X^*(E)$ be an ample character. Then one has
$$H^{0}\left(\left[E \backslash U\right],\Lscr(\lambda)^{\otimes n}\right)=H^{0}\left(\left[E \backslash G\right],\Lscr(\lambda)^{\otimes n}\right)$$
for all $n\geq 1$. For $n=N_Ud$, $d\geq 1$, this space is one-dimensional and any nonzero element induces a function $G\to k$ which vanishes exactly on the complement of $U$. 
\end{theorem}

\begin{proof}
The natural pull-back map $H^{0}\left(\left[E \backslash G\right],\Lscr(\lambda)^{\otimes n}\right) \to H^{0}\left(\left[E \backslash U\right],\Lscr(\lambda)^{\otimes n}\right)$ is clearly injective. Since $\Pic^E(U)$ is finite, we have an isomorphism
$$\Ecal(U)_\QQ \simeq X^*(E)_\QQ.$$
The space $H^{0}\left(\left[E \backslash U\right],\Lscr(\lambda)^{\otimes n}\right)$ is one-dimensional if and only if the function in $\Ecal(U)_\QQ$ corresponding to $n \lambda$ is in $\Ecal(U) \subset \Ecal(U)_\QQ$. In this case the function extends to a regular function on $G$ which vanishes exactly outside $U$, by \cite{KoWd} Theorem 3.8.
\end{proof}

Now let us return to the notations of section 1. For an element $w\in \leftexp{J}{W}$, denote by $N_w$ the integer associated to the $E$-orbit $O^w$ as in Theorem \ref{H0}. The map (\ref{Zetaw}) induces by pullback a map
$$H^{0}\left(\left[E \backslash O^w\right],\Lscr(\lambda)^{\otimes N_w d}\right)\to H^{0}\left(S^w,\zeta^*\Lscr(\lambda)^{\otimes N_w d}\right)$$
As explained in \cite{KoWd} 4.6 and in the proof of Theorem 4.12, there is a character $\lambda_{\omega_S}$ of $E$ such that
$$\zeta^*\Lscr(\lambda_{\omega_S})=\omega_S.$$
Thus we get a non-vanishing section $H_w$ of $\omega_S^{N_w d}$ over $S_w$ (well-defined up to a scalar), which proves Theorem 1. Note that this construction depends only on the choice of the Siegel embedding. Therefore we call $H_w$ a canonical Hasse invariant for $S_w$. It is a difficult question whether or not $H_w$ extends to the closure $\overline{S^w}$. In an upcoming article, we will show this for Hilbert-Blumenthal Shimura varieties and some unitary cases.

\subsection{Functoriality}

Let $f:G_1 \to G_2$ be a morphism of connected reductive groups over $\FF_p$. Let $\mu_1:\GG_{m,k}\to G_{1,k}$ be a minuscule cocharacter, and set $\mu_2:=f\circ \mu_1$. For $i=1,2$, denote by $(G_i,P_i,Q_i,\varphi)$ the zip datum attached to $\mu_i$. Denote by $E_1$ and $E_2$ respectively the corresponding zip groups. The map $f$ induces naturally a map $E_1\to E_2$, which we denote again by $f$. We get a map of stacks:
$$[E_1\backslash G_{1,k}]\longrightarrow [E_2 \backslash G_{2,k}].$$ 
Let $C_1$ be an $E_1$-orbit in $G_{1,k}$ and let $C_2$ be the $E_2$-orbit containing $f(C_1)$. Let $\lambda\in X^*(E_2)$ be a character of $E_2$ and denote by $N_1(\lambda \circ f)$ and $N_2(\lambda)$ the integers attached to the pairs $(C_1,\lambda \circ f)$ and $(C_2,\lambda)$ as in Remark \ref{rmkN}. We get an map:
$$\tilde{f}:H^{0} \left( \left[E_2 \backslash C_2\right],\Lscr(\lambda)^{\otimes N_2(\lambda)} \right) \to H^{0} \left(\left[E_1 \backslash C_1\right],\Lscr(\lambda\circ f)^{\otimes N_2(\lambda)} \right)$$
One sees readily that this map is injective. Since the space on the left has dimension one, we deduce that it is an isomorphism. In particular the integer $N_1(\lambda \circ f)$ divides $N_2(\lambda)$.

Assume now that $f$ is an embedding and that $C_1$ is the open $E_1$-orbit in $G_1$. Define again $C_2$ to be the $E_2$-orbit containing $f(C_1)$. Let $\lambda \in X^*(E_2)$ be an ample character. Then $\lambda \circ f$ is again ample (Remark 3.5 in loc. cit.). We deduce the following isomorphism:
$$H^{0} \left( \left[E_2\backslash C_2\right],\Lscr(\lambda)^{\otimes N_2(\lambda)} \right) \simeq H^{0} \left(\left[E_1 \backslash G_1\right],\Lscr(\lambda\circ f)^{\otimes N_2(\lambda)} \right).$$
Any nonzero element $H$ in this space induces a function $H:G_1\to k$ which vanishes exactly outside $C_1$ by Theorem \ref{mainthm}. But by definition it does not vanish on the preimage of $C_2$, so we get the following:

\begin{proposition}\label{orbs}
Assume that $f$ is an embedding. Let $C_1$ denote the open $E_1$-orbit in $G$, and let $C_2$ be the $E_2$-orbit containing $C_1$. Then we have the following:
$$C_1=f^{-1}(C_2).$$
\end{proposition}

\section{Consequences}

\subsection{Nonemptiness of Ekedahl-Oort strata for projective Shimura varieties of Hodge-type}

In this paragraph we assume that S is a projective variety. We show that all Ekedahl-Oort strata must be nonempty. Since the map $\zeta$ is open, it suffices to prove that the superspecial locus is nonempty, that is, the unique Ekedahl-Oort stratum of dimension zero.

\begin{lemma} \label{amplesec}
Assume that $S$ is projective. Let $S^w$ be an Ekedahl-Oort stratum of positive dimension. Then $S^w$ is not closed.
\end{lemma}

\begin{proof}
Since $S^w$ is quasi-affine, it cannot be projective unless it is zero-dimensional.
\end{proof}

We deduce that for any nonempty Ekedahl-Oort stratum of dimension $d \geq 1$, there is a nonempty Ekedahl-Oort stratum of dimension $ < d$ in its closure. Using this argument
recursively (begining at the $\mu$-ordinary stratum, which is nonempty), we deduce that there is a nonempty Ekedahl-Oort stratum of dimension $0$, which must be the superspecial one. This concludes the proof.

\begin{remark}\
\begin{enumerate}[(a)]
\item This proof will work for any Shimura variety of Hodge-type, provided that the closure of an Ekedahl-Oort stratum in the minimal compactification $S_K^{\rm min}$ intersects the boundary in a closed subset of codimension $\geq 2$.
\item It follows from \cite{Ni} that every Newton stratum contains a fundamental Ekedahl-Oort stratum. The argument is completely group-theoretic, and hence applies also to Hodge type Shimura varieties. We deduce that all Newton strata of a projective Shimura variety of Hodge type are nonempty as well.
\end{enumerate}
\end{remark}

\subsection{Embeddings of Shimura varieties}

We set $\overline{\Lambda}:=\Lambda \otimes_{\ZZ_p} \FF_p$, endowed with the symplectic form induced by $\psi$. We denote by $G_0$ the group $\GSp(\bar\Lambda)$. As explained in \cite{KoWd} 4.5, the embedding (\ref{EqEmbGroup}) induces a commutative diagram:
$$\xymatrix@1@M3pt@C6pc{
S_K \ar[d] \ar[r]^-{\zeta} &  \GZip^{\chi} \ar[d]^-\iota \\
S_{\tilde{K}}({\rm GSp}(\Lambda),S^{\pm})_\kappa \ar[r]^-{\zeta_0} & G_0-{\tt Zip}^{\iota\circ\chi}
}$$
where $S_{\tilde{K}}({\rm GSp}(\Lambda),S^{\pm})$ denotes the special fiber of the Siegel modular variety $\Scal_{\Ktilde}(\GSp(\Lambda),S^{\pm})$ and $\zeta_0$ the corresponding zip map. See diagram (4.9) in loc. sit. for details. The cocharacter $\iota\circ \chi$ of $G_0$ gives rise to a zip datum $(G_0,P_0,Q_0,\varphi)$ and we get a map between the quotient stacks:
$$[E\backslash G]\longrightarrow [E_0 \backslash G_0].$$
Denote by $U$ the open $E$-orbit in $G$ and by $U_0$ the $E_0$-orbit containing $\iota(U)$. We deduce from Proposition \ref{orbs} that $U=\iota^{-1}(U_0)$, and the following result follows:

\begin{corollary3} 
Let $S^\mu$ be the $\mu$-ordinary Ekedahl-Oort stratum, and let $S_0^\mu$ be the Ekedahl-Oort stratum of the Siegel modular variety $S_{\tilde{K}}({\rm GSp}(\Lambda),S^{\pm})$ containing the image of $S^\mu$ by the embedding $\iota:S_K\to S_{\tilde{K}}({\rm GSp}(\Lambda),S^{\pm})$. Then one has the equality:
$$S^\mu=\iota^{-1}(S_0^\mu).$$
In particular the $\mu$-ordinary locus $S^\mu$ is entirely determined in $S$ by the isomorphism class of the $p$-divisible group $A[p^\infty]$ (or by the group scheme $A[p]$), where $A$ is the abelian variety over $k$ attached to a $k$-point in $S_K$.
\end{corollary3}

\end{document}